\documentclass[12pt]{amsart}
\usepackage{amsfonts,amssymb,amsthm,eucal,amsmath,verbatim}
\allowdisplaybreaks
\newtheorem{thm}{Theorem}

\newtheorem{lemma}[thm]{Lemma}

\newcommand{\R}{\mathbb{R}}
\newcommand{\Z}{\mathbb{Z}}
\newcommand{\E}{\mathbb{E}}

\newcommand{\T}{\mathbb{T}}
\newcommand{\V}{\mathcal{V}}
\newcommand{\C}{\mathbb{C}}

\newcommand{\inprod}[2]{\left\langle #1, #2 \right\rangle}

\renewcommand{\P}{\mathbb{P}}

\renewcommand{\L}{\mathcal{L}}
\newcommand{\I}{\mathbb{I}}

\newcommand{\var}{\mathrm{Var}}

\newcommand{\hess}{\mathrm{Hess\,}}
\newcommand{\tr}{\mathrm{Tr\,}}
\newcommand{\n}{\mathfrak{N}}
\newcommand{\Hess}{\mathrm{Hess\,}}
\renewcommand{\O}{\mathcal{O}}
\newcommand{\U}{\mathcal{U}}
\newcommand{\dv}{d\overline{\rm vol}}

\begin{document}

\title{On Stein's method for multivariate normal approximation}
\author{Elizabeth S.\ Meckes}
\address{Department of Mathematics\\Case Western Reserve University\\
10900 Euclid Ave.\\Cleveland, OH 44106.}
\email{ese3@cwru.edu}
\urladdr{http://case.edu/artsci/math/esmeckes/}
\begin{abstract}The purpose of this paper is to synthesize the approaches taken by 
Chatterjee-Meckes and Reinert-R\"ollin in adapting Stein's method of exchangeable pairs
for multivariate normal approximation.  The more general linear regression condition of 
Reinert-R\"ollin allows for wider applicability of the method, while the method of bounding the
solution of the Stein equation due to Chatterjee-Meckes allows for improved convergence rates.  
Two abstract normal approximation theorems are proved, one for use when the underlying symmetries
of the random variables are discrete, and one for use in contexts in which continuous symmetry
groups are present.  The application to runs on the line from Reinert-R\"ollin is reworked to 
demonstrate the improvement in convergence rates, and a new application to joint value distributions
of eigenfunctions of the Laplace-Beltrami operator on a compact Riemannian manifold is 
presented.
\end{abstract}
\maketitle

\section{Introduction}

In 1972, Charles Stein \cite{Ste1} introduced a powerful new method for
estimating the distance from a probability distribution on $\R$ to a Gaussian
distribution.  Central to the method was the notion of a characterizing 
operator: Stein observed that the standard normal distribution was the 
unique probability distribution $\mu$ with the property that 
$$\int\big[f'(x)-xf(x)\big]\mu(dx)=0$$
for all $f$ for which the left-hand side exists and is finite.  The operator
$T_o$ defined on $C^1$ functions by 
$$T_of(x)=f'(x)-xf(x)$$
is called the characterizing operator of the standard normal distribution.
The left-inverse to $T_o$, denoted $U_o$, is defined by the equation
$$T_o(U_of)(x)=f(x)-\E f(Z),$$
where $Z$ is a standard normal random variable; the boundedness properties
of $U_o$ are an essential ingredient of Stein's method.

Stein and many other authors continued to develop this method; in 1986, Stein
published the book \cite{Ste2},
which laid out his approach to the method, called the method
of exchangeable pairs, in detail.  Stein's method has proved very useful
in situations in which local dependence or weak global dependence are
present.  One of the chief advantages of the method is that it is specifically
a method for bounding the distance from a fixed distribution to Gaussian, and
thus automatically produces concrete error bounds in limit theorems.  The
method is most naturally formulated by viewing probability measures as dual
to various classes of functions, so that the notions of distance that arise
are those which can be expressed as differences of expectations of test
functions (e.g., the total variation distance, Wasserstein distance, or
bounded Lipschitz distance).  Several authors (particularly 
Bolthausen \cite{Bol}, G\"otze \cite{Got},
Rinott and Rotar \cite{RinRot}, and Shao and Su \cite{ShaSu}) 
have extended the method to non-smooth
test functions, such as indicator functions of intervals in $\R$ and
indicator functions of convex sets in $\R^k$.

Heuristically, the univariate method of exchangeable pairs goes as follows.
Let $W$ be a random variable conjectured to be approximately Gaussian; 
assume that $\E W=0$ and $\E W^2=1.$  From $W$, construct a new random variable
$W'$ such that the pair $(W,W')$ has the same distribution as $(W',W)$.  
This is usually done by making a ``small random change'' in $W$, so that 
$W$ and $W'$ are close.  Let $\Delta=W'-W$.  If it can be verified that
there is a $\lambda>0$ such that 
\begin{equation}\label{lin-cond-univ}
\E\left[\Delta\big|W\right]=-\lambda W+E_1,
\end{equation}
\begin{equation}
\E\left[\Delta^2\big|W\right]=2\lambda+E_2,
\end{equation}
\begin{equation}
\E\left[\Delta\big|W\right]=E_3,
\end{equation}
with the random quantities $E_1,E_2,E_3$ being small compared to $\lambda$,
then $W$ is indeed approximately Gaussian, and its distance to Gaussian
(in some metric) can be bounded in terms of the $E_i$ and $\lambda$.

While there had been successful uses of multivariate versions of Stein's
method for normal approximation in the years following the introduction of 
the univariate method 
(e.g., by G\"otze \cite{Got}, Rinott and Rotar \cite{RinRot}, \cite{RinRot2}, 
and Rai{\v{c}} \cite{Rai}), there had not until recently been a version of
the method of exchangeable pairs for use in a multivariate setting.  This was first addressed in
joint work by the author with S. Chatterjee \cite{MecCha}, where several abstract normal approximation
theorems, for approximating by standard Gaussian random vectors, were proved.  The theorems 
were applied to estimate the rate of convergence in
 the multivariate central limit theorem and to show that rank $k$ projections of Haar measure on
 the orthogonal group $\O_n$ and the unitary group $\U_n$ are close to Gaussian measure on
 $\R^k$ (respectively $\C^k$), when $k=o(n)$.  The condition in the theorems of \cite{MecCha}
 corresponding to condition \eqref{lin-cond-univ} above was that, for an exchangeable pair of
 random vectors $(X,X')$,
 \begin{equation}\label{lin-cond-multi}
 \E\left[X'-X\big|X\right]=-\lambda X.
 \end{equation}
 The addition of a random error to this equation was not needed in the applications in \cite{MecCha},
 but is a straightforward modification of the theorems proved there.
 
 After the initial draft of \cite{MecCha} appeared on the ArXiv, a preprint was posted by 
 Reinert and R\"ollin \cite{ReiRol} which generalized one of the abstract normal approximation
 theorems of \cite{MecCha}.  Instead of condition \eqref{lin-cond-multi} above, they required
 \begin{equation}\label{lin-cond-matrix}
 \E\left[X'-X\big|X\right]=-\Lambda X+E,
 \end{equation}
 where $\Lambda$ is a positive definite matrix and $E$ is a random error.  This more general
 condition allowed them to estimate the distance to Gaussian random vectors with
 non-identity (even singular) covariance matrices.  They then introduced an insightful new method,
``the embedding method" for approximating real random variables by the normal distribution, by
 observing that in many cases in which the condition \eqref{lin-cond-univ} does not hold, the
 random variable in question can be viewed as one component of a random vector which 
 satisfies condition \eqref{lin-cond-matrix} with a non-diagonal $\Lambda$.  Many examples are
 given, both of the embedding method and the multivariate normal approximation theorem
 directly, including applications to runs on the line, statistics of Bernoulli random graphs, U-statistics,
 and doubly-indexed permutation statistics.
 
 After \cite{ReiRol} was posted, \cite{MecCha} underwent significant revisions, largely to change the
 metrics which were used on the space of probability measures on $\R^k$ and $\C^k$.  As 
 mentioned above, Stein's method works most naturally to compare measures by using (usually
 smooth) classes of test functions.  The smoothness conditions used by Reinert and R\"ollin, and
 those initially used in \cite{MecCha}, are to assume bounds on the quantities
 $$|h|_r:=\sup_{1\le i_1,\ldots,i_r\le k}\left\|\frac{\partial^r h}{\partial x_{i_1}\cdots\partial x_{i_r}}
 \right\|_{\infty}.$$
 The approach taken in the published version of \cite{MecCha} is to give smoothness conditions
 instead by requiring bounds on the quantities
 $$M_r(h):=\sup_{x\in\R^k}\|D^rh(x)\|_{op},$$
 where $\|D^rh(x)\|_{op}$ is the operator norm of the $r$-th derivative of $h$, as an $r$-linear form.
 These smoothness conditions seem preferable for several reasons.  Firstly, they are more 
 geometrically natural, as they are coordinate-free; they depend only on distances and not on
 the choice of orthonormal basis of $\R^k$.  Particularly when approximating by the standard 
 Gaussian distribution on $\R^k$, which is of course rotationally invariant, it seems desirable to have
 a notion of distance which is also rotationally invariant.  In more practical terms, considering
 classes of functions defined in terms of bounds on the quantities $M_r$ and modifying the
 proofs of the abstract theorems accordingly allows for improved error bounds.  The original 
bound on the Wasserstein distance from a $k$-dimensional projection of Haar
 measure on $\O_n$ to standard Gauss measure from the first version of  \cite{MecCha}  was
 $c\frac{k^{3/2}}{n}$, while the coordinate-free viewpoint allowed the bound to  be improved to
 $c\frac{k}{n}$ (in the same metric).  In Section \ref{examples} below, the example of runs on
 the line from \cite{ReiRol} is reworked with this viewpoint, with essentially the same ingredients,
 to demonstrate that the rates of convergence obtained are improved.  Finally, most of the bounds in
 \cite{MecCha} and below, and those from the main theorem in \cite{ReiRol} require two or 
 three derivatives, so that an additional smoothing argument is needed to move to one of the more
 usual metrics on probability measures (e.g. Wasserstein distance, total variation distance, or bounded
 Lipschitz distance).  Starting from bounds in terms of the $M_r(h)$ instead of the $|h|_r$
 typically produces better results in the final metric; compare, e.g., Proposition 3.2 of the original ArXiv 
 version of the paper \cite{MMec} of M. Meckes with Corollary 3.5 of the published version, in which 
 one of the abstract approximation theorems of 
 \cite{MecCha} was applied to the study of the distribution of marginals of the uniform measure on
 high-dimensional convex bodies.
 
 The purpose of this paper is to synthesize the approaches taken by the author and Chatterjee in
 \cite{MecCha} and Reinert and R\"ollin in \cite{ReiRol}.  In Section \ref{theorems}, two preliminary
 lemmas are proved, identifying a characterizing operator for the Gaussian distribution on $\R^k$
 with covariance matrix $\Sigma$ and bounding the derivatives of its left-inverse in terms of the
 quantities $M_r$.  Then, two abstract normal approximation theorems are proved.  The first is a
 synthesis of Theorem 2.3 of \cite{MecCha} and Theorem 2.1 of \cite{ReiRol}, in which the 
 distance from $X$ to a Gaussian random variable with mean zero and covariance $\Sigma$ is bounded, for $X$ the 
 first member of an exchangeable pair $(X,X')$ satisfying condition \eqref{lin-cond-matrix} above.
 The second approximation theorem is analogous to Theorem 2.4 of \cite{MecCha}, and is for situations
 in which the underlying random variable possesses ``continuous symmetries".  A condition 
 similar to \eqref{lin-cond-matrix} is used in that theorem as well.  Finally, in Section \ref{examples}, 
 two applications are carried out.  The first is simply a reworking of the runs on the line example of
 \cite{ReiRol}, making use of their analysis together with Theorem \ref{abstract} below to obtain a
 better rate of convergence.  The second application is to the joint value distribution of a finite
 sequence of orthonormal eigenfunctions of the Laplace-Beltrami operator on a compact Riemannian
 manifold.  This is a multivariate version of the main theorem of \cite{Mec1}.  As an example, the
 error bound of this theorem is computed explicitly for a certain class of flat tori.

\subsection{Notation and conventions}

The Wasserstein distance 
$d_{W}(X,Y)$ between the random variables $X$ and $Y$ is defined by
$$d_{W}(X,Y)=\sup_{M_1(g)\le1}\big|\E g(X)-\E g(Y)\big|,$$
where $M_1(g)=\sup_{x\neq y}\frac{|g(x)-g(y)|}{|x-y|}$ is the Lipschitz 
constant of $g$.  
  On the space of 
probability distributions
with finite absolute first moment, Wasserstein
distance induces a stronger topology than the usual 
one described by weak convergence, but not as strong as 
the topology induced by the total variation distance.  See \cite{dud} for 
detailed discussion of the various notions of distance between probability
distributions.

We will use $\n(\mu,\Sigma)$ to denote the normal distribution on $\R^k$ with 
mean $\mu$ and covariance matrix $\Sigma$; unless otherwise stated, the 
random variable $Z=(Z_1,\ldots,Z_k)$ is  
understood to be a standard Gaussian random vector on $\R^k$.  

In $\R^n$, the Euclidean inner product is denoted $\inprod{\cdot}{\cdot}$ and
the Euclidean norm is denoted $|\cdot|$.  
On the space of real $n\times n$ matrices, 
the Hilbert-Schmidt inner
product is defined by
$$\inprod{A}{B}_{H.S.}=\tr(AB^T),$$
with corresponding norm
$$\|A\|_{H.S.}=\sqrt{\tr(AA^T)}.$$
The operator norm of a matrix $A$ over $\R$ is defined by
$$\|A\|_{op}=\sup_{|v|=1,|w|=1}|\inprod{Av}{w}|.$$
More generally, if $A$ is a $k$-linear form on $\R^n$, the operator norm
of $A$ is defined to be 
$$\|A\|_{op}=\sup\{|A(u_1,\ldots,u_k)|:|u_1|=\cdots=|u_n|=1 \}.$$
The $n\times n$ identity matrix is denoted $I_n$ and the 
$n\times n$ matrix of all zeros is denoted $0_n$.

For $\Omega$ a domain in $\R^n$, the notation 
$C^k(\Omega)$ will be used for the space of $k$-times 
continuously differentiable real-valued functions on $\Omega$, and 
$C^k_o(\Omega)\subseteq C^k(\Omega)$ are those $C^k$ functions on 
$\Omega$ with compact support.  The $k$-th derivative $D^kf(x)$ of 
a function $f\in C^k(\R^n)$ is a $k$-linear form on $\R^n$, given in
coordinates by
$$\inprod{D^kf(x)}{(u_1,\ldots,u_k)}=\sum_{i_1,\ldots,i_k=1}^n \frac{\partial^k
f}{\partial x_{i_1}\cdots\partial x_{i_k}}(x)(u_1)_{i_1}\cdots(u_k)_{i_k},$$
where $(u_i)_j$ denotes the $j$-th component of the vector $u_i$.  For
an intrinsic, coordinate-free developement, see Federer \cite{Fed}.
For $f:\R^n\to\R$, sufficiently smooth, let
\begin{equation}\label{Mkdef}
M_k(f):=\sup_{x\in\R^n}\|D^kf(x)\|_{op}.
\end{equation}
In the case $k=2$, define 
\begin{equation}\label{Mtildedef}
\widetilde M_2(f):=\sup_{x\in\R^n}\|\Hess f(x)\|_{H.S.}.
\end{equation}
Note also that
$$M_k(f)=\sup_{x\neq y}\frac{\|D^{k-1}f(x)-D^{k-1}f(y)\|_{op}}{|x-y|};$$
that is, $M_k(f)$ is the Lipschitz constant of the $k-1$-st derivative of
$f$.

This general definition of $M_k$ is a departure from what was done by
Rai{\v{c}} in
\cite{raic}; there, smoothness conditions on functions are also given
in coordinate-independent ways, and 
$M_1$ and $M_2$ are defined as they are here, but in  
case $k=3$, the quantity $M_3$
is defined as the Lipschitz constant of the Hessian with respect to
the Hilbert-Schmidt norm as opposed to the operator norm.

\section{Abstract Approximation Theorems}\label{theorems}
This section contains the basic lemmas giving the Stein characterization of
the multivariate Gaussian distribution and bounds to the solution of the 
Stein equation, together with two multivariate abstract normal approximation
theorems and their proofs.  The first theorem is a reworking of the theorem
of Reinert and R\"ollin on multivariate normal approximation with the 
method of exchangeable pairs for vectors with non-identity covariance.  The
second is an analogous result in the context of ``continuous symmetries'' of
the underlying random variable, as has been previously studied by the author
in \cite{Mec2}, \cite{Mec1}, and (jointly with S. Chatterjee) in \cite{MecCha}.

\medskip

 The following lemma gives a second-order
characterizing operator for the Gaussian distribution with mean
$0$ and covariance $\Sigma$ on $\R^d$.   The characterizing operator for this distribution is
already well-known.  The proofs available in the literature generally rely on viewing the Stein 
equation in terms of the generator of the Ornstein-Uhlenbeck semi-group; the proof given here is
direct.
\begin{lemma}\label{char}
Let $Z\in\R^d$ be a random vector with $\{Z_i\}_{i=1}^d$ independent, 
identically distributed standard Gaussian random variables, and let $Z_\Sigma=
\Sigma^{1/2}Z$ for a symmetric, non-negative definite matrix $\Sigma$.  
\begin{enumerate}
\item \label{char1}
If $f:\R^d\to\R$ is two times continuously differentiable and 
compactly supported, then 
$$\E\big[\inprod{\hess f(Z_\Sigma)}{\Sigma}_{H.S.}-\inprod{Z_\Sigma}{
\nabla f(Z_\Sigma)}\big]=0.$$

\item \label{char2}If $Y\in\R^d$ is a random vector such that 
$$\E\big[\inprod{\hess f(Y)}{\Sigma}_{H.S.}-\inprod{Y}{\nabla f(Y)}\big]=0$$
for every $f\in C^2(\R^d)$ with $\E\big|\inprod{\hess f(Y)}{\Sigma}_{H.S.}
-\inprod{Y}{\nabla f(Y)}\big|<\infty$, then $\L(Y)=\L(Z_\Sigma)$.  

\item \label{sol}If $g\in C^\infty(\R^d)$, then the function
\begin{equation}\label{U_o}
U_og(x):=\int_0^1\frac{1}{2t}\big[\E g(\sqrt{t}x+\sqrt{1-t}Z_\Sigma)-
\E g(Z_\Sigma)\big]
dt\end{equation}
is a solution to the differential equation
\begin{equation}\label{diffeq}
\inprod{x}{\nabla h(x)}-\inprod{\hess h(x)}{\Sigma}_{H.S.}=g(x)-\E g(Z_\Sigma).
\end{equation}
\end{enumerate}

\end{lemma}

\begin{proof}
Part \eqref{char1} follows from integration by parts.  

Part \eqref{char2} follows easily from part \eqref{sol}: note that if 
$\E\big[\inprod{\hess f(Y)}{\Sigma}_{H.S.}-\inprod{Y}{\nabla 
f(Y)}\big]=0$
for every $f\in C^2(\R^d)$ with $\E\big|\inprod{\hess f(Y)}{\Sigma}_{H.S.}
-\inprod{Y}{\nabla 
f(Y)}\big|<\infty$, then for $g\in C_o^\infty$ given,
$$\E g(Y)-\E g(Z)=\E\big[\inprod{\hess (U_og)(Y)}{\Sigma}_{H.S.}
-\inprod{Y}{\nabla (U_og)(Y)}\big]=0,$$
and so $\L(Y)=\L(Z)$ since $C^\infty$ is dense in the class of bounded
continuous functions, with respect to the supremum norm.

For part \eqref{sol}, first note that since $g$ is Lipschitz, if 
$t\in\left(0,1\right)$
\begin{equation*}\begin{split}
\left|\frac{1}{2t}\big[\E g(\sqrt{t}x+\sqrt{1-t}\Sigma^{1/2}Z)-
\E g(\Sigma^{1/2}Z)\big]\right|&\le
\frac{L}{2t}\E\left|\sqrt{t}x+(\sqrt{1-t}-1)\Sigma^{1/2}Z\right|\\
&\le\frac{L}{2t}\left[\sqrt{t}|x|+t\sqrt{\tr(\Sigma)}\right],
\end{split}\end{equation*}
which is integrable on $\left(0,1\right)$, 
so the integral exists by the dominated
convergence theorem.  

  To show that $U_og$ is indeed a solution
to the differential equation (\ref{diffeq}), let $$Z_{x,t}=\sqrt{t}x+
\sqrt{1-t}\Sigma^{1/2}Z$$
and observe that
\begin{eqnarray*}
g(x)-\E g(\Sigma^{1/2}Z)&=&\int_0^1\frac{d}{dt}\E g(Z_{x,t})dt\\
&=&\int_0^1\frac{1}{2\sqrt{t}}\E(x\cdot\nabla g(Z_t))dt-\int_0^1\frac{1}
{2\sqrt{1-t}}\E\inprod{\Sigma^{1/2}Z}{\nabla g(Z_t)}dt\\&=&
\int_0^1\frac{1}{2\sqrt{t}}\E(x\cdot\nabla g(Z_t))dt-\int_0^1\frac{1}
{2}\E\inprod{\hess g(Z_t)}{\Sigma}_{H.S.}dt
\end{eqnarray*}
by integration by parts.  
Noting that 
\begin{equation*}
\hess(U_og)(x)=\int_0^1\frac{1}{2}\E\big[\hess g(Z_t)\big]dt
\end{equation*}
and
\begin{equation*}
x\cdot\nabla(U_og)(x)=\int_0^1\frac{1}{2\sqrt{t}}\E(x\cdot\nabla g(Z_t))dt
\end{equation*}
completes part \ref{sol}.

\end{proof}

The next lemma gives useful bounds on $U_og$ and its derivatives
in terms of $g$ and its derivatives.  
As in \cite{raic},  bounds are most 
naturally given
in terms of the 
quantities $M_i(g)$ defined in the introduction.
\begin{lemma}\label{bounds}
For $g:\R^d\to\R$ given, $U_og$ satisfies the following bounds:
\begin{enumerate}
\item \label{genbd}$$M_k(U_og)\le\frac{1}{k}M_k(g)\qquad\forall k\ge1.$$

\item \label{Hessbd1}$$\widetilde{M}_2(U_og)\le \frac{1}{2}
\widetilde{M}_2(g).$$

\hspace{-.5in}If, in addition, $\Sigma$ is positive definite, then

\bigskip

\item\label{gradbd}$$M_1(U_og)\le M_o(g)\|\Sigma^{-1/2}\|_{op}\sqrt{\frac{
\pi}{2}}.$$

\item \label{Hessbd2}$$\widetilde{M}_2(U_og)\le \sqrt{\frac{2}{
\pi}}M_1(g)\|\Sigma^{-1/2}\|_{op}.$$

\item\label{M3bd}$$M_3(U_og)\le\frac{\sqrt{2\pi}}{4}M_2(g)\|\Sigma^{-1/2}\|_{
op}.$$

\end{enumerate}
\end{lemma}

\medskip

{\bf Remark:} Bounds \eqref{gradbd}, \eqref{Hessbd2}, and \eqref{M3bd}
are mainly of use when $\Sigma$ has a fairly simple form, since they require
an estimate for $\|\Sigma^{-1/2}\|_{op}$.  They are also of theoretical interest,
since they show that if $\Sigma$ is non-singular, then 
the operator $U_o$ is smoothing; functions $U_og$ are
typically one order smoother than $g$.  The bounds \eqref{genbd} and 
\eqref{Hessbd1}, while not showing the smoothing behavior of $U_o$,
are useful when $\Sigma$ is complicated (or singular) and an estimate of $\|\Sigma^{-1/2}
\|_{op}$ is infeasible or impossible.

\medskip

\begin{proof}[Proof of Lemma \ref{bounds}]
Write $h(x)=U_og(x)$ and $Z_{x,t}=\sqrt{t}x+\sqrt{1-t}\Sigma^{1/2}Z$.  
Note that by the formula for $U_og$,
\begin{equation}\label{derivs}
\frac{\partial^rh}{\partial x_{i_1}\cdots\partial x_{i_r}}(x)=
\int_0^1(2t)^{-1}t^{r/2}\E\left[\frac{\partial^rg}{\partial x_{i_1}
\cdots\partial x_{i_r}}(Z_{x,t})\right]dt.
\end{equation} 
Thus
$$\inprod{D^k(U_og)(x)}{(u_1,\ldots,u_k)}=\int_0^1\frac{t^{\frac{k}{2}-1}}{2}\E
\big[\inprod{D^kg(Z_{x,t})}{(u_1,\ldots,u_k)}\big]dt$$
for unit vectors $u_1,\ldots,u_k$, and part \eqref{genbd} follows immediately.

For the second part, note that \eqref{derivs} implies that
$$\hess h(x)=\frac{1}{2}\int_0^1\E\left[\hess g(Z_{x,t})\right]dt.$$
Fix a $d\times d$ matrix $A$.  Then
\begin{equation*}\begin{split}
\left|\inprod{\hess h(x)}{A}_{H.S.}\right|&\le\frac{1}{2}\int_0^1\E
\left|\inprod{\hess g(Z_{x,t})}{A}_{H.S.}\right|dt\le\frac{1}{2}
\left(\sup_x\|\hess g(x)\|_{H.S.}\right)\|A\|_{H.S.},
\end{split}\end{equation*}
hence part \eqref{Hessbd1}.

For part \eqref{gradbd}, note that it follows by integration by parts on the 
Gaussian expectation that
\begin{equation*}\begin{split}
\frac{\partial h}{\partial x_i}(x)&=\int_0^1\frac{1}{2\sqrt{t}}\E\left[\frac{
\partial g}{\partial x_i}(\sqrt{t}x+\sqrt{1-t}\Sigma^{1/2}Z)\right]dt\\&=
\int_0^1\frac{1}{2\sqrt{t(1-t)}}\E\left[(\Sigma^{-1/2}Z)_ig(\sqrt{t}x+\sqrt{1-
t}\Sigma^{1/2}Z)\right]dt,
\end{split}\end{equation*}
thus $$\nabla h(x)=\int_0^1\frac{1}{2\sqrt{t(1-t)}}\E\left[g(Z_{x,t})
\Sigma^{-1/2}Z\right]dt,$$
and so $$M_1(h)\le\|g\|_{\infty}\E\big|\Sigma^{-1/2}Z\big|\int_0^1\frac{1}{2
\sqrt{t(1-t)}}dt.$$
Now, $\E\big|\Sigma^{-1/2}Z\big|\le\|\Sigma^{-1/2}\|_{op}\E|Z_1|=
\|\Sigma^{-1/2}\|_{op}\sqrt{\frac{2}{\pi}},$ since $\Sigma^{-1/2}Z$ is a 
univariate Gaussian random variable, and $\int_0^1\frac{1}{2\sqrt{t(1-t)}}=
\frac{\pi}{2}$.  This completes part \eqref{gradbd}.

For part \eqref{Hessbd2}, again using integration by parts on the 
Gaussian expectation,
\begin{equation}\begin{split}\label{seconds}
\frac{\partial^2h}{\partial x_i\partial x_j}(x)&=\int_0^1\frac{1}{2}\E\left[
\frac{\partial^2g}{\partial x_i\partial x_j}(\sqrt{t}x+\sqrt{1-t}\Sigma^{1/2}
Z)\right]dt\\
&=\int_0^1\frac{1}{2\sqrt{1-t}}\E\left[\big[\Sigma^{-1/2}Z\big]_i
\frac{\partial g}{\partial x_j}(Z_{x,t})
\right]dt,
\end{split}\end{equation}
and so
\begin{equation}\label{Hessian}
\hess h(x)=\int_0^1\frac{1}{2\sqrt{1-t}}\E\left[\Sigma^{-1/2}Z\left(\nabla g(
Z_{x,t})\right)^T\right]dt.\end{equation}

Fix a $d\times d$ matrix $A$.  Then
\begin{equation*}\begin{split}
\inprod{\hess h(x)}{A}_{H.S.}&=\int_0^1\frac{1}{2\sqrt{1-t}}\E\left[\inprod{
A^T\Sigma^{-1/2}Z}{\nabla g(Z_{x,t})}\right]dt,
\end{split}\end{equation*}
thus
$$\left|\inprod{\hess h(x)}{A}_{H.S.}\right|\le
M_1(g)\E|A^T\Sigma^{-1/2}Z|\int_0^1\frac{1}{2\sqrt{1-t}}dt=M_1(g)\E|A^T
\Sigma^{-1/2}Z|.$$
As above,
$$\E|A^T\Sigma^{-1/2}Z|\le\|A^T\Sigma^{-1/2}\|_{op}\sqrt{\frac{2}{\pi}}\le
\sqrt{\frac{2}{\pi}}\|\Sigma^{-1/2}\|_{op}\|A\|_{H.S.}.$$
It follows that
$$\|\hess h(x)\|_{H.S.}\le \sqrt{\frac{2}{\pi}}M_1(g)\|\Sigma^{-1/2}\|_{op}$$
for all $x\in\R^d$, hence part \eqref{Hessbd2}.

For part \eqref{M3bd}, let $u$ and $v$ be fixed vectors in $\R^d$ 
with $|u|=|v|=1.$  Then it follows from \eqref{Hessian} that 
$$\inprod{\left(\hess h(x)-\hess h(y)\right)u}{v}=\int_0^1\frac{1}{2\sqrt{1-t}
}\E\left[\inprod{\Sigma^{-1/2}Z}{v}\inprod{\nabla g(Z_{x,t})-
\nabla g(Z_{y,t})}{u}\right]dt,$$
and so
\begin{equation*}\begin{split}
\left|\inprod{(\hess h(x)-\hess h(y))u}{v}\right|&\le |x-y|\,M_2(g)\,\E|
\inprod{Z}{\Sigma^{-1/2}v}|
\int_0^1\frac{\sqrt{t}}{2\sqrt{1-t}}dt\\&=|x-y|\,M_2(g)\big|\Sigma^{-1/2}v\big|
\frac{\sqrt{2\pi}}{4}\\&\le|x-y|\,M_2(g)\big\|\Sigma^{-1/2}\big\|_{op}
\frac{\sqrt{2\pi}}{4}.
\end{split}\end{equation*}
\end{proof}

\begin{thm}\label{abstract}
Let $(X,X')$ be an exchangeable pair of random vectors in $\R^d$.  Suppose that
there is an invertible matrix $\Lambda$, a symmetric, non-negative definite
matrix $\Sigma$, a random vector $E$ and a 
random matrix $E'$ such that
\begin{enumerate}
\item \label{lincond}
$$\E\left[X'-X\big|X\right]=-\Lambda X+\E\left[E\big|X\right]$$
\item \label{quadcond}
$$\E\left[(X'-X)(X'-X)^T\big|X\right]=2\Lambda\Sigma+\E\left[E'\big|X
\right].$$
\end{enumerate}
Then for $g\in C^3(\R^d)$,
\begin{equation}\begin{split}\label{bd1}
\big|\E g(X)-\E g(\Sigma^{1/2}Z)\big|&\le\|\Lambda^{-1}\|_{op}\left[
 M_1(g)\E|E|+\frac{1}{4}\widetilde{M}_2(g)
\E\|E'\|_{H.S.}
+\frac{1}{9}M_3(g)\E|X'-X|^3\right]\\&\le\|\Lambda^{-1}\|_{op}\left[
 M_1(g)\E|E|+\frac{\sqrt{d}}{4}M_2(g)
\E\|E'\|_{H.S.}
+\frac{1}{9}M_3(g)\E|X'-X|^3\right]
,\end{split}\end{equation}
where $Z$ is a standard Gaussian random vector in $\R^d$.

If $\Sigma$ is non-singular, then for $g\in C^2(\R^d)$,
\begin{equation}\begin{split}\label{bd2}
\big|\E g(X)-\E g(\Sigma^{1/2}Z)\big|\le M_1(g)&\|\Lambda^{-1}\|_{op}\left[
\E|E|+\frac{1}{2}\|\Sigma^{-1/2}\|_{op}\E\|E'\|_{H.S.}
\right]\\
&+\frac{\sqrt{2\pi}}{24}M_2(g)\|\Sigma^{-1/2}\|_{op}\|\Lambda^{-1}\|_{op}
\E|X'-X|^3.\end{split}\end{equation}

\end{thm}

\begin{proof}
Fix $g$, and let
$U_og$ be as in Lemma \ref{char}.  Note that it suffices to assume that
$g\in C^\infty(\R^d)$:  let $h:\R^d\to\R$ be a centered Gaussian density with
covariance matrix $\epsilon^2I_d$.  Approximate $g$ by 
$g*h$; clearly  $\|g*h-g\|_\infty\to0$ as $\epsilon
\to0$, and by Young's inequality, $M_k(g*h)\le M_k(g)$ for all $k\ge 1$.

For notational convenience, let $f=U_og$.  By the exchangeability of 
$(X,X')$, 
\begin{equation*}\begin{split}
0&=\frac{1}{2}\E\left[\inprod{\Lambda^{-1}(X'-X)}{\nabla f(X')+\nabla f(X)}
\right]\\&=\E\left[\frac{1}{2}[\inprod{\Lambda^{-1}(X'-X)}{\nabla f(X')-
\nabla f(X)}+\inprod{\Lambda^{-1}(X'-X)}{\nabla f(X)}\right]\\&=\E\left[
\frac{1}{2}\inprod{\hess f(X)}{\Lambda^{-1}(X'-X)(X'-X)^T}_{H.S.}+
\inprod{\Lambda^{-1}(X'-X)}{\nabla f(X)}+\frac{R}{2}\right],
\end{split}\end{equation*}
where $R$ is the error in the Taylor approximation.  By conditions 
\eqref{lincond} and \eqref{quadcond}, it follows that 
\begin{equation*}\begin{split}
0=\E&\left[\inprod{\hess f(X)}{\Sigma}_{H.S.}-\inprod{X}{\nabla f(X)}+
\frac{1}{2}\inprod{\hess f(X)}{\Lambda^{-1}E'}_{H.S.}
+\inprod{\nabla f(X)}{\Lambda^{-1}E}+\frac{R}{2}
\right];\end{split}\end{equation*}
that is (making use of the definition of $f$),
\begin{equation}\label{erroreq}
\E g(X)-\E g(\Sigma^{1/2}Z)=\E\left[\frac{1}{2}\inprod{\hess f(X)}{\Lambda^{-1}
E'}_{H.S.}+\inprod{\nabla f(X)}{\Lambda^{-1}E}+\frac{R}{2}\right].
\end{equation}
Next,
\begin{equation*}\begin{split}
\E\left|\frac{1}{2}\inprod{\hess f(X)}{\Lambda^{-1}E'
}_{H.S.}\right|&\le
\frac{1}{2}\left(\sup_{x\in\R^d}\|\hess f(x)\|_{H.S.}\right)
\|\Lambda^{-1}E'\|_{H.S.}\\&\le
\frac{1}{2}\left(\sup_{x\in\R^d}\|\hess f(x)\|_{H.S.}\right)
\|\Lambda^{-1}\|_{op}\|E'\|_{H.S.}\\&
\le\frac{1}{2}\|\Lambda^{-1}\|_{op}\|E'\|_{H.S.}\left(\min\left\{\frac{1}{2}\widetilde{M}_2(g),\sqrt{\frac{2}{\pi}}
M_1(g)\|\Sigma^{-1/2}\|_{op}\right\}\right),
\end{split}\end{equation*}
where the first line is by the Cauchy-Schwarz inequality, the second is by the 
standard bound $\|AB\|_{H.S.}\le\|A\|_{op}\|B\|_{H.S.},$ and the third uses 
the bounds \eqref{Hessbd1} and \eqref{Hessbd2} from Lemma \ref{bounds}.

Similarly,
\begin{equation*}\begin{split}
\E\left|\inprod{\nabla f(X)}{\Lambda^{-1}E}\right|
&\le M_1(f)\|\Lambda^{-1}\|_{op}\E\left|E\right|\\&\le 
\|\Lambda^{-1}\|_{op}\E\left|E\right|\left(\min
\left\{M_1(g),\sqrt{\frac{\pi}{2}}M_o(g)\|\Sigma^{-1/2}\|_{op}\right\}\right).
\end{split}\end{equation*}
Finally, by Taylor's theorem and Lemma \ref{bounds},
$$|R|\le\frac{M_3(f)}{3}\big|X'-X\big|^2\big|\Lambda^{-1}(X'-X)\big|\le
\frac{1}{3}\|\Lambda^{-1}\|_{op}\big|X'-X\big|^3\left(\min\left\{\frac{1}{3}
M_3(g),\frac{\sqrt{2\pi}}{4}M_2(g)\|\Sigma^{-1/2}\|_{op}\right\}\right).$$
The first bound of the theorem results from choosing the first term from
each minimum; the second bound results from the second terms.

\end{proof}

\begin{thm}\label{inf-abstract}
Let $X$ be a random vector in $\R^d$ and, for each $\epsilon\in(0,1)$, 
suppose that $(X,X_\epsilon)$  is an exchangeable pair.  Suppose that
there is an invertible matrix $\Lambda$, a symmetric, non-negative definite
matrix $\Sigma$, a random vector $E$, a 
random matrix $E'$, and a deterministic function $s(\epsilon)$ such that
\begin{enumerate}
\item \label{inf-lincond}
$$\frac{1}{s(\epsilon)}\E\left[X'-X\big|X\right]\xrightarrow[\epsilon\to0]{L_1}-\Lambda X+\E\left[E\big|X\right]$$
\item \label{inf-quadcond}
$$\frac{1}{s(\epsilon)}\E\left[(X'-X)(X'-X)^T\big|X\right]\xrightarrow[\epsilon\to0]{L_1(\|\cdot\|_{H.S.})}2\Lambda\Sigma+\E\left[E'\big|X
\right].$$
\item \label{inf-tricond}
For each $\rho>0$, 
$$\lim_{\epsilon\to0}\frac{1}{s(\epsilon)}\E\left[|X_\epsilon-X|^2\I(|X_\epsilon-X|^2>\rho)\right]=0.$$
\end{enumerate}
Then for $g\in C^2(\R^d)$,
\begin{equation}\begin{split}\label{inf-bd1}
\big|\E g(X)-\E g(\Sigma^{1/2}Z)\big|&\le\|\Lambda^{-1}\|_{op}\left[
 M_1(g)\E|E|+\frac{1}{4}\widetilde{M}_2(g)
\E\|E'\|_{H.S.}\right]\\&\le\|\Lambda^{-1}\|_{op}\left[
 M_1(g)\E|E|+\frac{\sqrt{d}}{4}M_2(g)
\E\|E'\|_{H.S.}\right]
,\end{split}\end{equation}
where $Z$ is a standard Gaussian random vector in $\R^d$.

Also, if $\Sigma$ is non-singular,
\begin{equation}\begin{split}\label{inf-bd2}
d_W(X,\Sigma^{1/2}Z)\le \|\Lambda^{-1}\|_{op}\left[
\E|E|+\frac{1}{2}\|\Sigma^{-1/2}\|_{op}\E\|E'\|_{H.S.}
\right]
.\end{split}\end{equation}

\end{thm}

\begin{proof}
Fix $g$, and let
$U_og$ be as in Lemma \ref{char}.  As in the proof of Theorem \ref{abstract}, 
it suffices to assume that
$g\in C^\infty(\R^d)$.

For notational convenience, let $f=U_og$.  Beginning as before,

\begin{equation}\begin{split}\label{expansion}
0&=\frac{1}{2s(\epsilon)}\E\left[\inprod{\Lambda^{-1}(X_\epsilon-X)}{
\nabla f(X_\epsilon)+\nabla f(X)}\right]\\&=\frac{1}{s(\epsilon)}\E\left[
\frac{1}{2}[\inprod{\Lambda^{-1}(X_\epsilon-X)}{\nabla f(X_\epsilon)
-\nabla f(X)}+\inprod{\Lambda^{-1}(X_\epsilon-X)}{\nabla f(X)}\right]\\&=
\frac{1}{s(\epsilon)}E\left[\frac{1}{2}\inprod{\hess f(X)}{\Lambda^{-1}(
X_\epsilon-X)(X_\epsilon-X)^T}_{H.S.}+
\inprod{\Lambda^{-1}(X_\epsilon-X)}{\nabla f(X)}+\frac{R}{2}\right],
\end{split}\end{equation}
where $R$ is the error in the Taylor approximation.  

Now, by Taylor's theorem, there exists a real number $K$ depending on $f$,
such that
\begin{equation*}\begin{split}
|R|&\le K\min\big\{|X_\epsilon-X|^2|\Lambda^{-1}(X_\epsilon-X)|,
|X_\epsilon-X||\Lambda^{-1}(X_\epsilon-X)|\big\}\\&\le K\|\Lambda^{-1}\|_{op}
\min\big\{|X_\epsilon-X|^3,|X_\epsilon-X|^2\big\}\end{split}\end{equation*}
Breaking up the expectation over the sets on which $|X_\epsilon-X|^2$ is
larger and smaller than a fixed $\rho>0$,
 \begin{equation*}\begin{split}
\frac{1}{s(\epsilon)}\E\big|R\big|&\le\frac{K\|\Lambda^{-1}\|_{op}}{s(\epsilon)}
\E\Big[|X_\epsilon-X|^3\I(|X_\epsilon-X|\le\rho)+|X_\epsilon-X|^2\I(|X_\epsilon
-X|>\rho)\Big]\\&\le \frac{K\|\Lambda^{-1}\|_{op}\rho\E\big|X_\epsilon-X
\big|^2}{s(\epsilon)}+\frac{K\|\Lambda^{-1}\|_{op}}{s(\epsilon)}\E\Big[|X_\epsilon-X|^2\I(|X'-X|>\rho)\Big].
\end{split}\end{equation*}
The second term tends to zero as $\epsilon\to0$ by condition \ref{inf-tricond};
condition \ref{inf-quadcond} implies that the first is bounded by $CK\|
\Lambda^{-1}\|_{op}\rho$ for 
a constant $C$ depending on the 
distribution of $X$.  It follows that 
$$\lim_{\epsilon\to0}\frac{1}{s(\epsilon)}\E\big|R\big|=0.$$

For the rest of \eqref{expansion},
\begin{equation*}\begin{split}
\lim_{\epsilon\to0}&
\frac{1}{s(\epsilon)}\E\left[\frac{1}{2}\inprod{\hess f(X)}{\Lambda^{-1}(
X_\epsilon-X)(X_\epsilon-X)^T}_{H.S.}+
\inprod{\Lambda^{-1}(X_\epsilon-X)}{\nabla f(X)}\right]\\&=
\E\left[\inprod{\hess f(X)}{\Sigma}_{H.S.}-\inprod{X}{\nabla f(X)}
+\frac{1}{2}\inprod{\hess f(X)}{\Lambda^{-1}E'}_{H.S.}+\inprod{\nabla f(X)}{
\Lambda^{-1}E}\right],\end{split}\end{equation*}
where conditions \eqref{inf-lincond} and \eqref{inf-quadcond} together with
the boundedness of $\Hess f$ and $\nabla f$ have been used. 
That is (making use of the definition of $f$),
\begin{equation}\label{inf-erroreq}
\E g(X)-\E g(\Sigma^{1/2}Z)=\E\left[\frac{1}{2}\inprod{\hess f(X)}{\Lambda^{-1}
E'}_{H.S.}+\inprod{\nabla f(X)}{\Lambda^{-1}E}\right].
\end{equation}
As in the proof of Theorem \ref{abstract},
\begin{equation*}\begin{split}
\E\left|\frac{1}{2}\inprod{\hess f(X)}{\Lambda^{-1}E'
}_{H.S.}\right|&
\le\frac{1}{2}\|\Lambda^{-1}\|_{op}\|E'\|_{H.S.}\left(\min\left\{\frac{1}{2}\widetilde{M}_2(g),\sqrt{\frac{2}{\pi}}
M_1(g)\|\Sigma^{-1/2}\|_{op}\right\}\right),
\end{split}\end{equation*}
and
\begin{equation*}\begin{split}
\E\left|\inprod{\nabla f(X)}{\Lambda^{-1}E}\right|
&\le 
\|\Lambda^{-1}\|_{op}\E\left|E\right|M_1(g).
\end{split}\end{equation*}
This completes the proof.

\end{proof}

\medskip

\noindent{\bf Remarks:} 

\begin{enumerate}
\item 
Note that the condition 
$$\begin{array}{ll}{\rm (3')}\qquad
 \lim_{\epsilon\to0}\frac{1}{s(\epsilon)}\E\big|X_\epsilon-X
\big|^3=0,&\phantom{{\rm (3')}\qquad
 \lim_{\epsilon\to0}\frac{1}{s(\epsilon)}\E\big|X_\epsilon-X
\big|^3=0}\end{array}$$
is stronger than condition (3) of Theorem \ref{inf-abstract} and may be
used instead; this is what is done in the application given in Section 
\ref{examples}.

\item In \cite{ReiRol}, singular covariance matrices are treated by comparing
to a nearby non-singular covariance matrix rather than directly.  However, 
this is not necessary as all the proofs except those explicitly involving
$\Sigma^{-1/2}$ go through for non-negative definite $\Sigma$.
\end{enumerate}

\section{Examples}\label{examples}

\subsection{Runs on the line}
The following example was treated by Reinert and R\"ollin \cite{ReiRol} as
an example of the embedding method.  It should be emphasized that showing
that the number of $d$-runs on the line is asymptotically Gaussian seems
infeasible with Stein's original method of exchangeable pairs because of
the failure of condition \eqref{lin-cond-univ} from the introduction, but 
in \cite{ReiRol}, the random variable of interest is embedded in a
random vector whose components can be shown to be jointly Gaussian by 
making use of the more general condition \eqref{lin-cond-matrix} of the 
introduction.
The example is reworked here making use of the analysis of \cite{ReiRol}
together with
 Theorem \ref{abstract}, yielding an improved rate of convergence.

Let $X_1,\ldots,X_n$ be independent $\{0,1\}$-valued random variables, with
$\P(X_i=1)=p$ and $\P(X_i=0)=1-p$.  For $d\ge1$, define the (centered)
number of $d$-runs as 
$$V_d:=\sum_{m=1}^n(X_mX_{m+1}\cdots X_{m+d-1}-p^d),$$
assuming the torus convention, namely that $X_{n+k}=X_k$ for any $k$.  For
this example, we assume that $d<\frac{n}{2}$.
To make an exchangeable pair, $d-1$ sequential elements of $X:=(X_1,\ldots,X_n)$
are resampled.  That is, let $I$ be a uniformly distributed element of 
$\{1,\ldots,n\}$ and let $X_1',\ldots,X_n'$ be independent copies of the $X_i$.
Let $X'$ be constructed from $X$ by replacing $X_I,\ldots,X_{I+d-2}$ with
$X_I',\ldots,X_{I+d-2}'$.  Then $(X,X')$ is an exhangeable pair, and, 
defining $V_i':=V_i(X)$ for $i\ge 1$, it is easy to see that
\begin{equation}\begin{split}\label{Vlindiff}
V_i'-V_i=-&\sum_{m=I-i+1}^{I+d-2}X_m\cdots X_{m+i-1}+\sum_{m=I+d-i}^{I+d-2}X_m'
\cdots X_{I+d-2}'X_{I+d-1}\cdots X_{m+i-1}\\&+\sum_{m=I}^{I+d-i-1}X_m'\cdots 
X_{m+i-1}'+\sum_{m=I-i+1}^{I-1}X_m\cdots X_{I-1}X_I'\cdots X_{m+i-1}',
\end{split}\end{equation}
where sums $\sum_a^b$ are taken to be zero if $a>b$.  It follows that
$$\E\left[V_i'-V_i\big|X\right]=-\frac{1}{n}\left[(d+i-2)V_i-2\sum_{k=1}^{i-1}
p^{i-k}V_k\right].$$
Standard calculations show that, for $1\le j\le i\le d,$
\begin{equation}\begin{split}\label{Vcov}
\E\big[V_iV_j\big]&=n\left[(i-j+1)p^i+2\sum_{k=1}^{j-1}p^{i+j-k}-(i+j-1)
p^{i+j}\right]\\&=np^i(1-p)\sum_{k=0}^{j-1}(i-j+1+2k)p^k.
\end{split}\end{equation}
In particular, it follows from this expression that $np^i(1-p)\le\E V_i^2\le
np^i(1-p)i^2$, suggesting the renormalized random variables
\begin{equation}\label{Wdef}
W_i:=\frac{V_i}{\sqrt{np^i(1-p)}}.\end{equation}
It then follows from \eqref{Vcov} that, for $1\le i,j\le d$, 
\begin{equation}\label{Wcov}
\sigma_{ij}:=\E\big[W_iW_j\big]=p^{\frac{|i-j|}{2}}\sum_{k=0}^{i\wedge j-1}(|i-j|
+1+2k)p^k,
\end{equation}
and from \eqref{Vlindiff} that if $W:=(W_1,\ldots,W_d)$, then
$\E\left[W'-W\big|X\right]=\Lambda W,$ where 
$$\Lambda=\frac{1}{n}\begin{bmatrix}d-1\\ 
    -2p^{\frac{1}{2}}&d&&&0\\
    \vdots && \ddots&&&\\
    -2p^{\frac{k-1}{2}} & \cdots & -2p^{\frac{1}{2}} & d+k-2\\
    \vdots&&&&\ddots\\
    -2p^{\frac{d-1}{2}} &&\cdots && -2p^{\frac{1}{2}} & 2(d-1)
    \end{bmatrix}.$$
Condition \eqref{lincond} of Theorem \ref{abstract} thus applies with $E=0$ and 
$\Lambda$ as above.

To apply Theorem \ref{abstract}, an estimate on $\|\Lambda^{-1}\|_{op}$ is
needed.  Following Reinert and R\"ollin, we make use of known estimates of
condition numbers for triangular matrices (see, e.g., the survey
of Higham \cite{Hig}).  First, write $\Lambda=:\Lambda_E\Lambda_D$, where
$\Lambda_D$ is diagonal with the same diagonal entries as $\Lambda$ and 
$\Lambda_E$ is lower triangular with diagonal entries equal to one and 
$(\Lambda_E)_{ij}=\frac{\Lambda_{ij}}{\Lambda_{jj}}$ for $i>j$. Note that 
all non-diagonal entries of $\Lambda_E$ are bounded in absolute value by
$\frac{2\sqrt{p}}{d-1}$. From Lemeire \cite{Lem}, this implies the bounds
$$\|\Lambda_E^{-1}\|_1\le\left(1+\frac{2\sqrt{p}}{d-1}\right)^{d-1}\qquad 
\mbox{and}
\qquad\|\Lambda_E^{-1}\|_\infty\le\left(1+\frac{2\sqrt{p}}{d-1}\right)^{d-1}.$$
From Higham, $\|\Lambda_E^{-1}\|_{op}\le\sqrt{\|\Lambda_E^{-1}\|_1\|
\Lambda_E^{-1}\|_\infty},$
thus $$\|\Lambda_E^{-1}\|_{op}\le\left(1+\frac{2\sqrt{p}}{d-1}\right)^{d-1}.$$
Trivially, $\|\Lambda_D^{-1}\|_{op}=\frac{n}{d-1}$, and thus 
\begin{equation}\label{lambdabd}
\|\Lambda^{-1}\|_{op}\le\frac{n}{d-1}\left(1+\frac{2\sqrt{p}}{d-1}\right)^{d-1}
\le\frac{ne^{2\sqrt{p}}}{d-1}\le\frac{15n}{d}.
\end{equation}

Now observe that, if condition \eqref{lincond} of Theorem \ref{abstract} is
satisfied with $E=0$, then it follows that $\E\left[(W'-W)(W'-W)^T\right]=
2\Lambda\Sigma,$ and thus we may take $$E':=\E\left[(W'-W)(W'-W)^T
-2\Lambda\Sigma\big|W\right].$$
It follows that
\begin{equation*}\begin{split}
\E\|E'\|_{H.S.}&\le\sqrt{\sum_{i,j}\E(E'_{ij})^2}=\sqrt{\sum_{i,j}\var\left(
\E\left[(W_i'-W_i)(W_j'-W_j)\big|W\right]\right)}.
\end{split}\end{equation*}
It was determined by Reinert and R\"ollin that 
$$\var\left(
\E\left[(W_i'-W_i)(W_j'-W_j)\big|W\right]\right)\le \frac{96d^5}{n^3p^{2d}
(1-p)^2},$$
thus
$$\E\|E'\|_{H.S.}\le \frac{4\sqrt{6}d^{7/2}}{n^{3/2}p^d(1-p)}.$$

Finally, note that 
$$\E|W'-W|^3\le\sqrt{d}\sum_{i=1}^d\E\big|W_i'-W_i\big|^3.$$
Reinert and R\"ollin showed that
$$\E\big|(W_i'-W_i)(W_j'-W_j)(W_k'-W_k)\big|\le\frac{8d^3}{n^{3/2}p^{3d/2}
(1-p)^{3/2}}$$
for all $i,j,k$, thus
$$\E|W'-W|^3\le\frac{8d^{9/2}}{n^{3/2}p^{3d/2}(1-p)^{3/2}}.$$
Using these bounds in inequality \eqref{bd1} from Theorem \ref{abstract} 
yields the 
following.
\begin{thm}
For $W=(W_1,\ldots,W_d)$ defined as in \eqref{Wdef} with $d<\frac{n}{2}$, 
$\Sigma=\big[\sigma_{ij}\big]_{i,j=1}^d$ given by \eqref{Wcov}, and $h\in C^3(
\R^d)$,
\begin{equation}\label{runs}
\big|\E h(W)-\E h(\Sigma^{1/2}Z)\big|\le 
\left[\frac{15\sqrt{6}d^3M_2(h)}{p^d(1-p)\sqrt{n}}+\frac{40d^{7/2}M_3(h)}{3
p^{3d/2}(1-p)^{3/2}\sqrt{n}}\right],
\end{equation}
where $Z$ is a standard $d$-dimensional Gaussian random vector.
\end{thm}

\medskip

{\bf Remarks:} Compare this result to that obtained in \cite{ReiRol}:
\begin{equation}
 \big|\E h(W)-\E h(\Sigma^{1/2}Z)\big|\le 
\frac{37 d^{7/2}|h|_2}{p^{d}(1-p)\sqrt{n}}
    +\frac{10 d^5|h|_3}{p^{3d/2}(1-p)^{3/2}\sqrt{n}},\end{equation}
where $|h|_2=\sup_{i,j}\left\|\frac{\partial^2h}{\partial x_i\partial x_j}
\right\|_\infty$ and $|h|_3=\sup_{i,j,k}\left\|\frac{\partial^3h}{\partial x_i
\partial x_j\partial x_k}\right\|_{\infty}$.

\subsection{Eigenfunctions of the Laplacian}
Consider a compact Riemannian manifold $M$ with metric $g$.  Integration
with respect to the normalized volume measure is denoted $\dv$, thus
$\int_M1\dv=1.$
For coordinates $\left\{\frac{\partial}{\partial 
x_i}\right\}_{i=1}^n$ on $M$, define 
\begin{equation*}
(G(x))_{ij}=g_{ij}(x)=\inprod{\left.\frac{\partial}{\partial x_i}\right|_x}
{\left.\frac{\partial}{\partial x_j}\right|_x},\qquad g(x)=\det(G(x)),
\qquad g^{ij}(x)=(G^{-1}(x))_{ij}.
\end{equation*}
Define the gradient $\nabla f$ of $f:M\to\R$ and the Laplacian
$\Delta_g f$ of $f$ by
\begin{equation*}
\nabla f(x)=\sum_{j,k}\frac{\partial f}{\partial x_j}g^{jk}
\frac{\partial }{\partial x_k},\qquad\qquad \Delta_g f(x)
=\frac{1}{\sqrt{g}}\sum_{j,k}\frac{\partial}{\partial x_j}
\left(\sqrt{g}g^{jk}\frac{\partial f}{\partial x_k}\right).
\end{equation*}
The function $f:M\to\R$ is an eigenfunction of $\Delta$ with eigenvalue
$-\mu$ if $\Delta f(x)=-\mu f(x)$ for all $x\in M$; it is known
(see, e.g., \cite{Cha}) that on a compact Riemannian manifold $M$, the 
eigenvalues of $\Delta$ form a sequence $0\ge-\mu_1\ge-\mu_2\ge
\ldots\searrow-\infty$.  Eigenspaces associated to different eigenvalues
are orthogonal in $L_2(M)$ and all eigenfunctions of $\Delta$ are elements
of $C^\infty(M)$.

Let $X$ be a uniformly distributed random point of $M$.  The value 
distribution of a function $f$ on $M$ is the distribution (on $\R$) of
the random variable $f(X)$.
In \cite{Mec1}, a general bound was given for the total variation distance 
between the value distribution of an eigenfunction and a Gaussian distribution, 
in terms of the eigenvalue and the gradient of $f$.  The proof made use 
of a univariate version of Theorem \ref{inf-abstract}.  Essentially the same
analysis is used here to prove a multivariate version of that theorem.

Let $f_1,\ldots,f_k$ be a sequence of orthonormal (in $L_2$) eigenfunctions of
$\Delta$ with corresponding eigenvalues $-\mu_i$ (some of the $\mu_i$ may be
the same if the eigenspaces of $M$ have dimension greater than 1).
Define the random vector $W\in\R^k$ by $W_i:=f_i(X)$.  We will apply
Theorem \ref{inf-abstract} to show that $W$ is approximately
distributed as a standard Gaussian random vector (i.e., $\Sigma=I_k$).  

For $\epsilon>0$,
an exchangeable pair $(W,W_\epsilon)$ is constructed from $W$ as follows.  
Given $X$, choose an element $V\in S_XM$ (the unit sphere of the tangent
space to $M$ at $X$) according to the uniform measure on $S_XM$, and let 
$X_\epsilon=\exp_X(\epsilon V).$  That is, pick a direction at random, and
move a distance $\epsilon$ from $X$ along a geodesic in that direction.  It
was shown in \cite{Mec1} that this construction produces an exchangeable pair
of random points of $M$; it follows that if $W_\epsilon:=(f_1(X_\epsilon),
\ldots,f_k(X_\epsilon))$, then $(W,W_\epsilon)$ is an exchangeable pair
of random vectors in $\R^k$.

 In order to identify $\Lambda$, $E$ and $E'$ so as to apply 
Theorem \ref{inf-abstract}, first 
let $\gamma:[0,
\epsilon] \to M$ be a constant-speed geodesic such that $\gamma(0)=X$, 
$\gamma(\epsilon)=
X_\epsilon$, and $\gamma'(0)=V$.  
Then applying Taylor's theorem on $\R$ to
the function $f_i\circ\gamma$ yields
\begin{equation}\begin{split}\label{taylor}
f_i(X_\epsilon)-f_i(X)&=\epsilon\cdot\left.\frac{d(f_i\circ\gamma)}{dt}
\right|_{t=0}+\frac{\epsilon^2}{2}\cdot\left.\frac{d^2(f_i\circ\gamma)}{d
t^2}\right|_{t=0}+O(\epsilon^3)\\
&=\epsilon\cdot d_Xf_i(V)+\frac{\epsilon^2}{2}
\cdot\left.\frac{d^2(f_i\circ\gamma)}{dt^2}
\right|_{t=0}+O(\epsilon^3),
\end{split}\end{equation}  
where the coefficient implicit in the $O(\epsilon^3)$ depends on $f_i$ and 
$\gamma$ and $d_xf_i$ denotes the differential of $f_i$ at $x$.  Recall that
$d_xf_i(v)=\inprod{\nabla f_i(x)}{v}$ for $v\in T_xM$ and the gradient $
\nabla f_i(x)$ defined as above.    
Now, for $X$ fixed, 
$V$ is distributed according to normalized Lebesgue measure on
$S_XM$ and $d_Xf_i$ is a linear functional on $T_XM$.  It follows that
$$\E\left[d_Xf_i(V)\big|X\right]=\E\left[d_Xf_i(-V)\big|X\right]=
-\E\left[d_Xf_i(V)\big|X\right],$$
thus $\E\left[d_Xf_i(V)\big|X\right]=0.$
This implies that
$$\lim_{\epsilon\to0}\frac{1}{\epsilon^2}\E\big[f_i(X_\epsilon)-f_i(X)\big|
X\big]$$
exists and is finite; we will take $s(\epsilon)=\epsilon^2.$
Indeed, it is well-known (see, e.g., Theorem 11.12 of \cite{gray}) that 
\begin{equation}\label{mean-value}
\lim_{\epsilon\to0}\frac{1}{\epsilon^2}\E\big[f_i(X_\epsilon)
-f_i(X)\big|X\big]=\frac{1}{2n}\Delta_gf_i(X)=\frac{-\mu_i}{2n}f_i(X)
\end{equation}
for $n=dim(M).$
It follows that
$\Lambda = \frac{1}{2n}diag(\mu_1,\ldots,\mu_k)$ and
$E'=0.$  
The expression $\E\left[W_\epsilon-W\big|W\right]$
satisfies the $L_1$ convergence requirement of Theorem \ref{inf-abstract}, since the 
$f_i$ are necessarily smooth and $M$ is compact.  Furthermore, it is 
immediate that $\|\Lambda^{-1}\|_{op}=2n\max_{1\le i\le k}\left(\frac{1}{\mu_i}
\right).$

For the second condition of Theorem \ref{inf-abstract}, it is necessary 
to determine
$$\lim_{\epsilon\to0}\frac{1}{\epsilon^2}\E\big[(W_\epsilon-W)_i(W_\epsilon-
W)_j\big|X\big]=\lim_{\epsilon\to0}\frac{1}{\epsilon^2}\E\big[(f_i(X_\epsilon)-
f_i(X))(f_j(X_\epsilon)-f_j(X))\big|X\big].$$
By the expansion (\ref{taylor}),
\begin{equation*}
\E\left[(f_i(X_\epsilon)-f_i(X))(f_j(X_\epsilon)-f_j(X))\big|X
\right]=\epsilon^2\E\left[(d_Xf_i(V))(d_Xf_j(V))\big|X\right]+O(\epsilon^3).
\end{equation*}
Choose coordinates $\left\{\frac{\partial}{\partial x_i}\right\}_{i=1}^n$
in a neighborhood of $X$ which are orthonormal at $X$.  Then
$$\nabla f(X)=\sum_i\frac{\partial f}{\partial x_i}\frac{\partial}
{\partial x_i},$$
for any function $f\in C^1(M)$, thus
\begin{equation*}\begin{split}
(d_xf_i(v))\cdot(d_xf_j(v))&=\inprod{\nabla f_i}{v}\inprod{\nabla f_j}{v}\\
&=\sum_{r=1}^n\frac{\partial f_i}{\partial x_r}(x)\frac{
\partial f_j}{\partial x_r}(x)v_r^2+
\sum_{r\neq s}\frac{\partial f_i}{\partial x_r}(x)\frac{\partial f_j}{
\partial x_s}(x)v_rv_s.
\end{split}\end{equation*}
Since $V$ is uniformly distributed on a Euclidean sphere, 
$\E[V_rV_s]=\frac{1}{n}\delta_{rs}$.   Making use of this fact yields
$$\lim_{\epsilon\to0}\frac{1}{\epsilon^2}\E\left[(d_Xf_i(V))(d_Xf_j(V))
\big|X\right]=
\frac{1}{n}\inprod{\nabla f_i(X)}{\nabla f_j(X)},$$
thus condition (2) is satisfied with 
$$E'=\frac{1}{n}\Big[\inprod{\nabla f_i(X)}{\nabla f_j(X)}\Big]_{i,j=1}^k-
2\Lambda.$$
(As before, the convergence requirement is satisfied since the $f_i$ are
smooth and $M$ is compact.)

By Stokes' theorem,
\begin{equation*}\begin{split}
\E\inprod{\nabla f_i(X)}{\nabla f_j(X)}&=-\E\big[f_i(X)\Delta_g f_j(X)\big]=
\mu_j\E\big[f_i(X)f_j(X)\big]=\mu_i\delta_{ij},
\end{split}\end{equation*}
thus
\begin{equation*}\begin{split}
\E\|E'\|_{H.S.}&=\frac{1}{n}\E\sqrt{\sum_{i,j=1}^k\Big[\inprod{
\nabla f_i(X)}{\nabla f_j(X)}-\E\inprod{\nabla f_i(X)}{\nabla f_j(X)}\Big]}
\end{split}\end{equation*}

Finally, (\ref{taylor}) 
gives immediately that 
$$\E\left[|W_\epsilon-W|^3\big|W\right]
=O(\epsilon^3),$$
(where the implicit constants depend on the $f_i$ and on $k$), thus condition
\eqref{inf-tricond} of Theorem \ref{inf-abstract} is satisfied.

All together, we have proved the following.

\begin{thm}\label{eigenfunctions}
Let $M$ be a compact Riemannian manifold and $f_1,\ldots,f_k$ an orthonormal
(in $L_2(M)$) sequence of eigenfunctions of the Laplacian on $M$, with 
corresponding eigenvalues $-\mu_i$.  Let
$X$ be a uniformly distributed random point of $M$.  Then if $W:=(f_1(X),
\ldots,f_k(X))$, 
$$d_W(W,Z)\le\left[\max_{1\le i\le k}\left(\frac{1}{\mu_i}\right)\right]
\E\sqrt{\sum_{i,j=1}^k\Big[\inprod{\nabla f_i(X)}{\nabla f_j(X)}-
\E\inprod{\nabla f_i(X)}{\nabla f_j(X)}\Big]}.$$
\end{thm}

\medskip

{\bf Example:} The torus.

In this example, Theorem \ref{eigenfunctions} is applied to 
the value distributions of eigenfunctions
on flat tori.  The class of functions considered here are random functions;
that is, they are linear combinations of eigenfunctions with random
coefficients.

Let $(M,g)$ be the torus $\T^n=\R^n/\Z^n$, with the metric
given by the symmetric positive-definite bilinear form $B$: 
$$(x,y)_B=\inprod{Bx}{y}.$$  With this metric, the Laplacian $\Delta_B$
on $\T^n$ is given by
$$\Delta_Bf(x)=\sum_{j,k}(B^{-1})_{jk}\frac{\partial^2f}{\partial x_j\partial
x_k}(x).$$
Eigenfunctions of $\Delta_B$ are given by the real and imaginary parts
of functions of the form 
$$f_v(x)=e^{2\pi i\inprod{v}{x}_B}=e^{2\pi i\inprod{Bv}{x}},$$
for vectors $v\in\R^n$ such that $Bv$ has integer components, with
corresponding eigenvalue $-\mu_v=-(2\pi \|v\|_B)^2.$

Consider a collection of $k$ random eigenfunctions $\{f_j\}_{j=1}^k$ 
of $\Delta_B$ on the torus which
are linear combinations of eigenfunctions with random
coefficients:
$$f_j(x):=\Re\left(\sum_{v\in\V_j}a_ve^{2\pi i\inprod{Bv}{x}}\right),$$
where $\V_j$ is a finite collection of vectors $v$ such that 
$Bv$ has integer components and $\inprod{v}{Bv}=\frac{\mu_j}{(2\pi)^2}$ 
for each $v\in\V_j$, and
$\{\{a_v\}_{v\in\V_j}:1\le j\le k\}$ are $k$ independent random vectors (indexed by $j$)
on the spheres of radius
$\sqrt{2}$ in $\R^{|\V_j|}$.  Assume that $v+w\neq 0$ for $v\in\V_r$ and $w\in
\V_s$ ($r$ and $s$ may be equal) and that
$\V_r\cap\V_s=\emptyset$ for $r\neq s$; it follows easily that the $f_j$
are orthonormal in $L_2(\T^n)$.

To apply Theorem \ref{eigenfunctions}, first note that   
\begin{equation*}\begin{split}
\nabla_Bf_r(x)&=\left\{\Re\left(\sum_{j=1}^n\sum_{v\in\V_r}(2\pi i)a_v(Bv)_j
(B^{-1})_{j\ell}e^{2\pi i\inprod{Bv}{x}}\right)\right\}_{\ell=1}^n\\
&=-\Im\left(\sum_{v\in\V_r}(2\pi)a_ve^{2\pi i\inprod{Bv}{x}}v\right),
\end{split}\end{equation*}
using the fact that $B$ is symmetric.

It follows that
\begin{equation}\begin{split}
\inprod{\nabla_Bf_r(x)}{\nabla_Bf_s(x)}_B&=\sum_{j,\ell=1}^n
B_{j\ell}\Im\left(\sum_{v\in\V_r}(2\pi)a_v
e^{2\pi i\inprod{Bv}{x}}v_j\right)\Im\left(\sum_{w\in\V_s}(2\pi)a_we^{2\pi i
\inprod{Bw}{x}}w_\ell\right)\\&=\frac{1}{2}\Re\left[\sum_{\substack{
v\in\V_r\\w\in\V_s}}4\pi^2a_va_w\inprod{v}{w}_B\left(e^{2\pi i\inprod{Bv-Bw}{x}}-
e^{2\pi i\inprod{Bv+Bw}{x}}\right)\right].\label{grad2}
\end{split}\end{equation}
Let $X$ be a randomly distributed point on the torus.  Let $\E_a$ denote 
averaging over the coefficients $a_v$ and $\E_X$ denote averaging over the
random point $X$.
To estimate $\E_a d_W(W,Z)$ from Theorem \ref{eigenfunctions}, first apply the 
Cauchy-Schwartz inequality and then change the order of integration:
\begin{equation*}\begin{split}
\E_a\E_X&\sqrt{\sum_{i,j=1}^k\big[\inprod{\nabla f_i(X)}{\nabla f_j(X)}_B-
\E_X\inprod{\nabla f_i(X)}{\nabla f_j(X)}_B\big]}\\&\le
\sqrt{\sum_{i,j=1}^k\E_X\E_a\big[\inprod{\nabla f_i(X)}{\nabla f_j(X)}_B-
\E_X\inprod{\nabla f_i(X)}{\nabla f_j(X)}_B\big]^2}.\end{split}\end{equation*}
Start by computing $\E_X\E_a\inprod{\nabla_Bf_r(X)}{\nabla_Bf_s(X)}^2_B.$  
From above,
\begin{equation*}\begin{split}
&\hspace{-.3in}
\inprod{\nabla_Bf_r(x)}{\nabla_Bf_s(X)}_B^2\\&=2\pi^4\Re\left[\sum_{\substack{
v,v'\in\V_r\\w,w'\in\V_s}}a_va_wa_{v'}a_{w'}
\inprod{v}{w}_B\inprod{v'}{w'}_B\right.\\&
\qquad\qquad\qquad\Big[e^{2\pi i\inprod{Bv-Bw-Bv'+Bw'}{x}}-e^{2\pi i\inprod{
Bv-Bw-Bv'-Bw'}{x}}+e^{2\pi i\inprod{Bv-Bw+Bv'-Bw'}{x}}\\&
\qquad\qquad\qquad-e^{2\pi i\inprod{
Bv-Bw+Bv'+Bw'}{x}}-e^{2\pi i\inprod{Bv+Bw-Bv'+Bw'}{x}}+e^{2\pi i\inprod{
Bv+Bw-Bv'-Bw'}{x}}\\&\left.\qquad\qquad\qquad\phantom{\sum_v}
-e^{2\pi i\inprod{Bv+Bw+Bv'-Bw'}{x}}+e^{2\pi i\inprod{
Bv+Bw+Bv'+Bw'}{x}}\Big]\right].
\end{split}\end{equation*}
Averaging over the coefficients $\{a_v\}$ using standard techniques (see
Folland \cite{Fol} for general formulae and \cite{Mec1} for a detailed 
explanation of the univariate version of this result), and then over
the random point $X\in\T^n$, it is not hard to show that

\begin{equation*}\begin{split}
\E_X\E_a\|\nabla_Bf_r(X)\|_B^4&=\frac{8\pi^4}{|\V_r|(|\V_r|+2)}\left[3
\sum_{v\in\V_r}\|v\|_B^4+2\left(\sum_{v\in\V_r}\|v\|_B^2\right)^2+4
\sum_{v,w\in\V_r}\inprod{v}{w}_B^2\right],
\end{split}\end{equation*}
and 
\begin{equation*}
\E_X\E_a\inprod{\nabla_Bf_r(X)}{\nabla_Bf_s(X)}_B^2=
\frac{4\pi^4}{|\V_r||\V_s|}\sum_{\substack{v\in\V_r\\w\in\V_s}}\inprod{v}{w}_B^2.
\end{equation*}

Now,
\begin{equation*}\begin{split}
\E_a\Big[\E_X\|\nabla_Bf_r(X)\|_B^2\Big]^2&=
\E_a\left[2\pi^2\sum_{v\in\V_r}a_v^2\|v\|_B^2\right]^2=\frac{(2\pi)^4}{|\V_r|(|
\V_r|+2)}\left[\left(\sum_{v\in\V_r}\|v\|_B^2\right)^2+2\sum_{v\in\V_r}\|v\|_B^4
\right],
\end{split}\end{equation*}
and
$$\E_X\inprod{\nabla_B f_r(X)}{\nabla_B f_s(X)}_B=0$$
for $r\neq s$.
It follows that 
\begin{equation*}\begin{split}
\E_X\E_a\|\nabla_Bf_r(X)\|_B^4-\E_a\left(\E_X\|\nabla_Bf_r(X)\|_B^2\right)^2
\le\frac{2(2\pi)^4}{|\V_r|(|\V_r|+2)}\sum_{v,w\in\V_r}\inprod{v}{w}_B^2,
\end{split}\end{equation*}
and, applying Theorem \ref{eigenfunctions}, we have shown that 
\begin{thm}\label{torus}
Let the random orthonormal set of functions $\{f_r\}_{r=1}^k$ be defined on 
$\T^n$ as above, and let the random vector $W$ be defined by $W_i:=
f_i(X)$  for $X$ a random point of $\T^n$.  Then 
$$\E_ad_W(W,Z)\le\frac{4\pi^2}{\min_r \mu_r}\sqrt{\sum_{r,s=1}^k\left(\frac{2}{|\V_r||\V_s|}\sum_{
\substack{v\in\V_r\\w\in\V_s}}\inprod{v}{w}_B^2\right)}.$$
\end{thm}

\medskip

{\bf Remarks:} Note that if the elements of $\cup_{r=1}^k \V_r$ are mutually
orthogonal, then the right-hand side becomes
$$\frac{4\pi^4}{\min_r \mu_r}\sqrt{\sum_{r=1}^k\frac{2\mu_r}{|\V_r|^2}},$$
thus if it is possible to choose the $\V_r$ such that their sizes are 
large for large $n$, and the range of the $\mu_r$ is not too big, the 
error is small.  One can thus find vectors of orthonormal eigenfunctions of
$\T^n$ which are jointly Gaussian (and independent) in the limit as the 
dimension tends to infinity, if the
matrix $B$ is such that there are large collections of vectors $v$
which are ``close to orthogonal'' and have the same lengths
with respect to $\inprod{\cdot}{\cdot}_B$ and with 
the vectors $Bv$ having integer components.  It is possible to extend the
analysis here, in a fairly straightfoward manner, to require rather less of the
matrix $B$ (essentially all the conditions here can be allowed to hold 
only approximately), but for simplicity's sake, we include only this 
most basic version here.  The univariate version of this relaxing of 
conditions is carried out in detail in \cite{Mec1}.

\medskip

\noindent{\bf Acknowledgements. }The author thanks M.\ Meckes for many useful 
discussions.  This research was supported by an American Institute of 
Mathematics five-year fellowship.

\bibliographystyle{plain}
\bibliography{multi-survey}

\end{document}